\documentclass[a4paper,12pt]{amsart}
\usepackage{amsthm,amsmath,amsfonts,amssymb}

\usepackage{color}

\newlength{\hchng}
\newlength{\vchng}
\newtheorem{thm}{Theorem}[section]
\newtheorem{prop}[thm]{Proposition}

\newtheorem{lemma}[thm]{Lemma}

\newtheorem{preremark}[thm]{Remark}
\newenvironment{remark}{\begin{preremark}\rm}{\medskip \end{preremark}}
\numberwithin{equation}{section}

\newcommand{\R}{\mathbb R}

\begin{document}

\title[Local solvability and loss of smoothness  of the full MHD system]{Local solvability and loss of smoothness  of the Navier-Stokes-Maxwell equations
with  large initial data}
\date{}
\author{Slim Ibrahim  and Tsuyoshi  Yoneda}
\email{S. Ibrahim: ibrahim@math.uvic.ca} 
\urladdr{http://www.math.uvic.ca/~ibrahim/}
\email{T. Yoneda: yoneda@math.sci.hokudai.ac.jp}
\thanks{S. I. is partially supported by NSERC\# 371637-2009 grant and a start up fund from University
of Victoria}
\thanks{T. Y. is partially supported by PIMS Post-doc fellowship at the University
of Victoria, and partially supported by NSERC\# 371637-2009}

\maketitle
\begin{center}
Department of Mathematics and Statistics,
University of Victoria \\
PO Box 3060 STN CSC,
Victoria, BC, Canada, V8W 3R4 \\
\end{center}
\begin{center}
and
\end{center}
\begin{center}
Department of Mathematics, Hokkaido University\\ 
Sapporo 060-0810, Japan
\end{center}

\bibliographystyle{plain}
\noindent{\bf Abstract:}
Existence of  local-in-time unique  solution  and loss of smoothness of  full Magnet-Hydro-Dynamics system (MHD)
is considered for periodic initial data. The result is proven using Fujita-Kato's method in $\ell^1$ based (for the Fourier coefficients) functional spaces enabling us to easily estimate nonlinear terms in the system as well as solutions to Maxwells's equations. A loss of smoothness result is shown for the velocity and magnetic field.  It comes from the damped-wave operator which does not have any smoothing effect.




\section{Introduction}

In this paper we study  the following full Magnet-Hydro-Dynamics system (MHD):
\begin{equation}\label{MHD}
\begin{cases}
\partial_t v+v\cdot \nabla v-\nu\Delta v+\nabla p=j\times B\\
\partial_t E-\nabla\times B=-j\\
\partial_t B+\nabla\times E=0\\
\text{div} v=\text{div} B=0\\
\sigma(E+v\times B)=j
\end{cases}
\end{equation}
with the initial data
\begin{equation*}
v|_{t=0}=v_0,\quad B|_{t=0}=B_0,\ E|_{t=0}=E_0.
\end{equation*}
Here $v$, $E$, $B$: $\mathbb{R}^+_t\times \mathbb{R}^3_x\to \mathbb{R}^3$ are vector fields defined on $\mathbb{R}^3$. The vector field
$v=(v_1,v_2,v_3)$ is the velocity of the fluid, $\nu$ its viscosity and the scalar function $p$ stands for the pressure. The vector fields $E$ and $B$ 
are the electric and magnetic fields of the fluid, respectively. And $j$ is the electric current expressed by Omn's law.
The force term $j\times B$ in the Navier-Stokes equations comes from Lorentz force under a quasi-neutrality assumption of the net charge carried 
by the fluid. Note that the pressure $p$ can be recovered from $v$ and $j\times B$ via an explicit Calderon-Zygmund type operator.
The second equation is the Amp\`ere-Maxwell equation for an electric field $E$. The third equation is nothing but Farady's law. 
For a detailed introduction to the MHD, we refer to Davidson \cite{Davidson01} and Biskamp \cite{Biskamp93}.

Concerning the Cauchy problem associated to \eqref{MHD}, multiply the Navier-Stokes equations by $v$, the Amp\`ere-Maxwell equations by $(B,E)^T$ and integrate (using the divergence free condition of the velocity) to get the following formal energy identity 
\begin{eqnarray*} 
\label{energy}
\frac12  \frac{d}{dt} \big[\|v\|_{L^2}^2 +\|B\|_{L^2}^2 +\|E\|_{L^2}^2
   \big]   + \|j\|_{L^2}^2  + \nu\|\nabla v \|_{L^2}^2 = 0 
\end{eqnarray*}
showing that both the viscosity and the electric resistivity dissipate energy. It also suggests that with initial data in $\big(L^2(\R^d)\big)^3$, one can expect to construct a global finite energy weak solution (\`a la Leray) . However, this intuitive expectation remains an interesting open problem for system \eqref{MHD} in both dimensions $d=2,3$ as an actual difficulty stands to derive some compactness especially for the magnetic field in virtue of the hyperbolicity of  Maxwell's equations. This problem is similar to the global weak solvability of Euler's equations for ideal fluid which also remain open. In \cite{Lions}, P. L.-Lions suggested the notion of dissipative solutions for Euler's equation. These are functions which satisfy the energy inequality but not necessarily the system of equations itself. However,  these solutions do coincide with any strong solution with the same initial data, if there is any. Dissipative solutions were also investigated for other MHD models as in \cite{Wu}. For strong solutions,  local wellposedness and small data global existence were investigated for other models of magneto-hydrodynamic as for example in \cite{Giga-Yoshida1} \cite{Giga-Yoshida2}, but the nonlinearities there are weaker than the one in system \eqref{MHD}. In the two dimensional case with $(v_0, E_0, B_0)\in L^2(\mathbb{R}^2)\times(H^s(\mathbb{R}^2))^2$ and $s>0$, Masmoudi \cite{M} proved the existence and uniqueness of global strong solutions to \eqref{MHD}. More recently,  Keraani and the first author \cite{IK} showed the existence of global strong solutions in both dimension two and three with small initial data in spaces as close as possible to the energy space. However, the authors were not able to construct local solutions for arbitrary large initial data.\\
Imposing more regularity on the initial electro-magnetic field, one can hope to solve \eqref{MHD}.  In this paper we show the existence of local solutions with three dimensional large periodic initial data. Similar argument is valid in the case of the whole space and almost periodic case. We also show that if the initial data is not smooth enough, the solution is not smooth enough. In this consideration we use the periodic structure. 

Using the divergence free property of $B$, one can easily verify that 
\begin{equation*}
\nabla\times(\nabla\times B)=-\Delta B, 
\end{equation*}
and therefore, the magnetic field $B$ satisfies an inhomogeneous damped wave equation.
Let $\tilde E$ and $\bar E$ be the divergence free and the gradient potential parts of the electric field $E$ i.e. 
$$
\tilde E:=\bold P E,\quad\mbox{and}\quad \bar E:=\nabla(-\Delta)^{-1}\text{div} E.
$$ 
Note that $E=\tilde E+\bar E$. In what follows, we use the following re-written equation coupled with damped wave equation:
\begin{equation}\label{re MHD}
\begin{cases}
\partial_t v+\bold P(v\cdot \nabla v)-\nu\Delta v
=\bold P(j\times B)\\
\partial_{tt}B-\Delta B+\partial_t B=\nabla\times(v\times B)\\
\partial_{tt}\tilde E-\Delta \tilde E+\partial_t \tilde E=-\partial_t(v\times B)\\
\partial_t \bar E+\bar E=-\nabla(-\Delta)^{-1}\text{div}(u\times B)\\
\text{div} v=\text{div} B=0
\end{cases}
\end{equation}
with the initial data
\begin{equation*}
v|_{t=0}=v_0,\quad B|_{t=0}=B_0,\ \partial_t B|_{t=0}=B_1,\  \tilde E|_{t=0}=\bold P E_0,\
\partial_t \tilde  E|_{t=0}:=\tilde E_1,\ \bar E|_{t=0}:=\bar E_0. 
\end{equation*}
The formulation given by equation \eqref{re MHD} is good for constructing a mild solutions.
Let $\mathcal L_1$ and $\mathcal L_2$ be the propagators associated to the Fourier multiplier functions
\begin{equation*}
\Phi_1(t,n):=e^{-t/2}\cos \left(\sqrt{|n|^2-1/4}t\right),\quad \Phi_2(t,n):=e^{-t/2}\frac{\sin \left(\sqrt{|n|^2-1/4}t\right)}{\sqrt{|n|^2-1/4}}.
\end{equation*} 
Mild solutions can be written as  
\begin{eqnarray}\label{mild MHD}
v=v(x,t)&=&e^{t\Delta}v_0-\int_0^te^{t\Delta}\bold P\left[\nabla\cdot(v\otimes v)-E\times B-(v\times B)\times B\right]ds\\
\nonumber
&:=&e^{t\Delta}v_0+M_1(v,v)+M_2(E,B)-M_3(v,B,B)\quad\text{for}\quad x\in\mathbb{T}^3,\\
\nonumber
B= B(x,t)&=&\mathcal L_1(t)B_0+\mathcal L_2(t)(B_0/2+B_1)+\int_0^t\mathcal L_2(t-s)\left[\nabla\times(v\times B)\right]ds\\
\nonumber
&:=&\mathcal L_1(t)B_0+\mathcal L_2(t)(B_0/2+B_1)+M_4(v,B)\quad\text{for}\quad x\in\mathbb{T}^3\\
\nonumber
\end{eqnarray}
\begin{eqnarray*}
\tilde E=\tilde E(x,t)&=&\mathcal L_1(t)\tilde E_0+\mathcal L_2(t)(\tilde E_0/2+\tilde E_1)-\int_0^t\mathcal L_2(t-s)\partial_s\bold P(v\times B)(s)ds\\
\nonumber
&=&\mathcal L_1(t)\tilde E_0+\mathcal L_2(t)(\tilde E_0/2+\tilde E_1)+  \mathcal L_2(t)(v_0\times B_0) + \int_0^t\mathcal (\partial_s\mathcal L_2)(t-s)\partial_s\bold P(v\times B)(s)ds\\
\nonumber
&:=&\mathcal L_1(t)\tilde E_0+\mathcal L_2(t)(\tilde E_0/2+\tilde E_1)+  \mathcal L_2(t)(v_0\times B_0)  -M_5(v,B)\quad\text{for}\quad x\in\mathbb{T}^3,
\nonumber
\end{eqnarray*}
and
\begin{eqnarray*}
\bar E&=&e^{-t}\bar E_0-\int_0^Te^{-(t-s)}\nabla(\Delta)^{-1}\text{div} (v\times B)\;ds\\
\nonumber
&:=&\bar E_0 -M_6(v,B)\quad\text{for}\quad x\in\mathbb{T}^3,
\nonumber
\end{eqnarray*}
where $B_1=\partial_t B_0=-\nabla\times E_0$.

We handle periodic functions so that any smooth function can be decomposed as 
\begin{equation*}
f(x):=\sum_{n\in\mathbb{Z}^3}\hat f(n)e^{i n\cdot x}.
\end{equation*}
Throughout this paper, we assume that the mean value of the initial magnetic field is zero, namely $\hat B_0(0)=0$.
The  mean value of $B$ is preserved along the original equation. Indeed, from 
\begin{equation*}
\partial_t B+\nabla\times E=0,
\end{equation*} 
we have 
\begin{equation*}
\partial_t \hat B(n,t)+in\times \hat E(n,t)=0
\end{equation*}
which for $n=0$ gives $\partial_t \hat B(0,t)=0$. If moreover $\hat B(0,0)=0$, then $\hat B(0,t)=0$ for any $t>0$.

In what follows, we consider the  mild solution \eqref{mild MHD}. 
To show local existence of \eqref{mild MHD}, we need to find a function space in which one can show closed estimates for \eqref{mild MHD}.
In this point of view, $L^\infty$ is one candidate of such function spaces since \eqref{mild MHD} has a tri-linear term $M_3(v,B,B)$.
More precisely, we can easily have the following estimate for $M_3$: 
\begin{equation}\label{good estimate}
\|M_3\|_\infty\leq t\|(v\times B)\times B\|_\infty\leq t\|v\|_\infty\|B\|^2_\infty.
\end{equation}
However the $L^\infty$ Lebesgue space is not suitable for the damped wave propagator $\mathcal L_1$ and $\mathcal L_2$.
This difficulty is essentially the same as the unboundedness of  singular integral operator in $L^\infty$. To overcome this, we introduce the weighted ``$\ell^1$ (in Fourier side) space
with $s$ weight given by
\begin{equation*}
X^s:=\{u\in\mathcal S'(\mathbb{T}^3): \|u\|_{X^s}=\left\|(1+ |\cdot|^2)^{s/2}\hat u\right\|_{\ell^1}<\infty\}.
\end{equation*}
We denote by $\|u\|:= \|u\|_{X^0}$.
It is well known that $X^0$ is an algebra which is continuously embedded in $BUC$, the space of  bounded uniformly continuous functions. For the periodic case, 
we can point out the following relationship between $X^s(\mathbb{T}^3)$ and the H\"older space $C^s(\mathbb{T}^3)$. 
Let $C^{s+0}(\mathbb{T}^3)$ be $C^{s+\epsilon}(\mathbb{T}^3)$ 
and $X^{s+0}$ be $X^{s+\epsilon}$
 for some $\epsilon>0$.
\begin{prop}\label{relation}
We have $u\in C^\infty(\mathbb{T}^3)$ if and only if $u\in \cap_{s\geq 1}X^s$. Moreover,  $C^{s+3/2+0}(\mathbb{T}^3)$ is embedded in  $X^s$.
\end{prop}
The proof follows from the following well known inequalities, 
\begin{equation*}
\|u\|_{X^s}\lesssim \|u\|_{H^{s+3/2+0}(\mathbb{T}^3)}\lesssim \|u\|_{B^{s+3/2+0}_{\infty,\infty}(\mathbb{T}^3)},
\end{equation*}
where $B^{s}_{p,q}(\mathbb{T}^3)$ is the  Besov space.

In this paper, we use the space  $X^0$. Since it is an algebra, one can easily estimate  $M_3$ like \eqref{good estimate}. In addition, one can suitably estimate the damped wave operators $\mathcal L_1$ and $\mathcal L_2$  in $X^0$.
This kind of function spaces has already been used for fluid equations by several authors (see for example \cite{GIMM2, GJMY, IY, Titi, Y}). 
However, in order to get a good representation of the electric field $E$ as in \eqref{re MHD}, we need to decompose it into divergence free $\tilde E$ and vector potential part $\bar E$. Observe that we have  
$$
\|E\|_{X^0}\lesssim\|\tilde E\|_{X^0}+\|\bar E\|_{X^0}\lesssim \|E\|_{X^0}.
$$
The following local solvability and the propagator of $C^\infty$ regularity  is the first main result.
\begin{thm}\label{local}
Let $(v_0,B_0, E_0)\in (X^0\times X^0\times X^0)$
 with $\text{div}\ v_0=\text{div}\ B_0=0$.
Then there is a local-in-time unique mild solution of \eqref{MHD}
\begin{equation*}
(v,B,E)\in C([0,T]: X^0\times X^0\times X^0).
\end{equation*}
Moreover, if  $(v_0, B_0, E_0)\in C^\infty\times C^\infty\times C^\infty$, then
\begin{equation*}
(v,B,E)\in C^\infty([0,T]\times\mathbb{R}^3).
\end{equation*}
\end{thm}

Next we investgate a ``loss of smoothness" result.  We can show that either the fluid flow or the  magnetic field 
should develop a $X^{5+0}$ singularity within the lifetime of the solution even if initially neither the velocity field nor the magnetic field does have such singularities.
  This phenomena is caused by the damped-wave operator which does not have any  smoothing effect (see  \cite{N} for example). 
We have the following second main result. 

\begin{thm}\label{loss of smoothness}
For sufficiently small $\delta>0$, any $E_{0,2}$, $E_{0,3}\in C^\infty$ and $(v_0, B_0) \in C^\infty \times  C^\infty $, there exist $E_{0,1}\in (X^1\setminus X^{1+\delta/2})$ and a unique solution of \eqref{MHD}on $[0,T]$ such that one can choose time $t\in(0,T)$ with the property that either of the following alternatives can  happen: 
  \begin{equation*}
  v(t)\not\in X^{5+\delta}(\mathbb{T}^3)\quad\text{or}\quad
   B(t)\not\in  X^{5+\delta}(\mathbb{T}^3).
   \end{equation*}

\end{thm}

\begin{remark}
For the usual 3D-Naveir-Stokes equations (which means $B_0=E_0=0$), the solution is smooth enough even if the initial data $v_0$ is in 
 $X^1\setminus X^{1+\delta/2}$.
\end{remark}
Now we give the rough idea to the proof of Theorem \ref{loss of smoothness}.
From the mild formulation of $B$ with $B_1=\nabla\times E_0$, one can get inflation from the term $\mathcal L_2 B_1$.
By choosing $B_0$ smooth enough, $\mathcal L_1 B_0$ and $\mathcal L_2 B_0$ remain smooth enough and if we assume that $u(t)$ and $B(t)$ 
remain smooth for time $t$, then $u\times B$ and $\mathcal L_2(t)\nabla(u\times B)$ will be smooth and therefore can be absorbed by the inflation term.
This leads to a contradiction. Thus, either term $u(t)$ or $B(t)$ must not be  sufficiently smooth.

\section{Proof of the main results}
\begin{proof} [Proof of Theorem \ref{local}.] Let us set 
\begin{equation*}
M_t:=\max\{\sup_{0\leq s<t}\|v\|,\sup_{0\leq s<t}\|E\|,\sup_{0\leq s<t}\|B\|\}
\end{equation*}
Direct calculation shows the following estimates for $M_j$ ($j=1,\cdots 6$):
\begin{eqnarray}\label{usual estimates}
\|M_1(v,v)\|&\leq& t^{1/2}M_t^2,
\\
\nonumber
\|M_2(E,B)\|&\leq& tM_t^2,\quad
\\
\nonumber
\|M_3(v,B,B)\|&\leq& tM_t^3,\quad
\\
\nonumber
\|M_4(v,B)\|&\leq& tM_t^2,
\\
\nonumber
\|M_5(v,B)\|&\leq& tM_t^2,\\
\nonumber
\|M_6(v,B)\|&\leq& tM_t^2.\\
\nonumber
\end{eqnarray}
Indeed, on the one hand since 
\begin{equation*}
\sup_{0\leq s\leq T}s^{1/2}|n|e^{-|n|^2s}\leq C,\quad \|\bold P v\|\leq C\|v\|,\quad
(v\cdot \nabla)v=\nabla\cdot(v\otimes v)\quad\text{and}\quad \|v\otimes v\|\lesssim \|v\|^2,
\end{equation*}
 we see that 
\begin{equation*}
\|M_1(v,v)\|\leq \int_0^t\frac{C}{(t-s)^{1/2}}\sup_{0\leq s'\leq T}\|v(s')\|^2ds
\end{equation*}
which clearly leads to the desired estimate of $M_1$.
Similarly, for $M_2$,  we see that
\begin{equation*}
\|M_2(E,B)\|\leq \int_0^t\sup_{0\leq s'\leq T}\|E(s')\|\|B(s')\|ds,
\end{equation*}
and this gives us the desired estimate for $M_2$.
The other terms are estimated similarly. On the other hand, we also have 
\begin{eqnarray*}
\| e^{t\Delta}u_0\|&\leq& \|u_0\|,
\\
\|\mathcal L_1(t)B_0\|&\lesssim &\|B_0\|,
\\
\| \mathcal L_2(t)B_0\|&\lesssim& \|B_0\|,
\\
\| \mathcal L_2(t)B_1\|&\lesssim& \|B_1\|,
\\
\|\mathcal L_1(t)\tilde E_0\|&\lesssim &\|\tilde E_0\|,
\\
\|\mathcal L_2(t)\tilde E_0\|&\leq &\|\tilde E_0\|,
\\
\|\mathcal L_2(t)\tilde E_1\|&\leq &\|\tilde E_1\|.
\end{eqnarray*}
By the above estimates, we can apply a usual fixed point argument. We sketch the proof in below and for full details refer for example to \cite{GIMM2}. 
Define the successive approximation $\{u_j\}_{j=1,2,\cdots}$  by
\begin{equation*}
u_1(t):=
\begin{pmatrix}
v_1\\
B_1\\
\tilde E_1\\
\bar E_1
\end{pmatrix}
:=
\begin{pmatrix}
e^{t\Delta}u_0\\
\mathcal L_1(t)B_0+\mathcal L_2(t)(B_0/2+B_1)\\
\mathcal L_1(t)\tilde E_0+\mathcal L_2(t)(\tilde E_0/2+\tilde E_1)
+\mathcal L_2(t)(v_0\times B_0)
\\
e^{-t}\bar E_0
\end{pmatrix}
\end{equation*}  
and 
\begin{equation*}
u_{j+1}(t):=
\begin{pmatrix}
v_{j+1}\\
B_{j+1}\\
\tilde E_{j+1}\\
\bar E_{j+1}
\end{pmatrix}
:=
\begin{pmatrix}
M_1(v_j,v_j)+M_2(E_j,B_j)-M_3(v_j,B_j,B_j)\\
M_4(v_j,B_j)\\
 -M_5(v_j,B_j)\\
 -M_6(v_j,B_j)
\end{pmatrix}.
\end{equation*}
Set
\begin{equation*}
K_j(T):=\sup_{0\leq s\leq T}\|v_j(s)\|+\sup_{0\leq s\leq T}\|E_j(s)\|+\sup_{0\leq s\leq T}\|\tilde E_j(s)\|+\sup_{0\leq s\leq T}\|\bar E_j(s)\|
\end{equation*}
and 
$$
L_{j+1}:=\sup_{0\leq s\leq T}\|u_{j+1}-u_j\|.
$$
Using estimates \eqref{usual estimates} we have
\begin{equation}
\label{est K}
K_{j+1}(T)\leq K_1+ C_1T^{1/2}K_j^2+C_2T K_j^2+C_3TK_j^3.
\end{equation}
with positive constants $C_\ell>0$, $\ell=1,2,3$. The remaining part is now standard. From estimate \eqref{est K}, one can easily derive the uniform bound 
$K_j(T)\leq 2K_1\lesssim(\|u_0\|+\|B_0\|+\|E_0\|)$ when  $T$ is sufficiently small with respect of the norm of the initial data.
Similar calculation gives the pointwise estimate on $L_j$
\begin{eqnarray*}
L_{j+1}(T)&\leq &\tilde C_1 T^{1/2}K_jL_j+\tilde C_2 T K_jL_j+\tilde C_3TK_j^2L_j\\
 &\leq &2\tilde C_1 T^{1/2}K_1L_j+2\tilde C_2 TK_1 L_j+2\tilde C_3TK_1^2L_j
\end{eqnarray*}
with positive constants $\tilde C_\ell>0$, $\ell=1,2,3$.
Now we take a sufficiently small $T$ such that $2\tilde C_1 T^{1/2}K_1<1/6$, $2\tilde C_2 TK_1<1/6$ and $2\tilde C_3TK_1^2<1/6$,
we have 
\begin{equation*}
L_{j+1}\leq(1/2)L_j.
\end{equation*}
This gives us that there is a unique limit $u$ such that $u_j\to u\in C([0,T]:X^0\times X^0\times X^0)$ as $j\to\infty$.
It is easy to see that the limit $u$ uniquely solves the equation.\\
To show that $(v,B,E)$ are smooth if $(v_0,B_0,E_0)$ are smooth enough,
The argument is, for example, similar to \cite[Proposition 15.1]{Le}.
From \eqref{mild MHD} and estimates \eqref{usual estimates}, we see that 
\begin{eqnarray}\label{second derivative}
\|\nabla v(t)\|&\leq &\|\nabla v_0\|+\int_0^t(t-s)^{-1/2}\bigg(\|v\|\|\nabla v\|+\| E\|\|B\|+\|v\|\|B\|^2\bigg)\;ds\\
\nonumber
\|\nabla B(t)\|&\leq &\|\nabla B_0\|+\|\nabla B_1\|+\int_0^t\bigg(\|\nabla v\|\|B\|+\|v\|\|\nabla B\|\bigg)\;ds\\
\nonumber
\|\nabla \tilde E(t)\|&\leq &\|\nabla \tilde E_0\|+\|\nabla \tilde E_1\|+\int_0^t\bigg(\|\nabla v\| \|B\|+\| v \| \|\nabla B\| \bigg)\;ds\\
\nonumber
\|\nabla \bar E(t)\|&\leq &\|\nabla \bar E_0\|+\int_0^t\bigg(\|\nabla v\| \|B\|+\| v \| \|\nabla B\| \bigg)\;ds.
\nonumber
\end{eqnarray}
By the above estimates \eqref{second derivative} and  Gronwall's inequality, we can control  $\|\nabla B\|+\|\nabla v\|+\|\nabla \tilde E\|+\|\nabla \bar E\|$ uniformly in $t\in [0,T]$
since $v,B,\tilde E,\bar E \in C([0,T]:X^0)$.
Repeating this argument for  $\|\nabla^{k} B\|+\|\nabla^{k} v\|+\|\nabla^k \tilde E\|+\|\nabla^k \bar E\|$ ($k\geq 2$), we can derive the space regularity of the solution.
We can then get the time regularity from original equation \eqref{MHD} once we have the space regularity.
\end{proof}

Before proving  Theorem \ref{loss of smoothness}, we need 
a pointwise estimate of $M_4$. For  that the following convolution estimate is the key.

\begin{lemma}\label{pointwiseconv}
For $n\in\mathbb{Z}^3$,
let $\rho_s(n):=C_1/(C_2+|n|^{3+s})$, for some arbitrary positive constants $C_1$ and $C_2$.
If $s, s'>0$, then
\begin{equation*}
|\sum_{k\in\mathbb{Z}^3} \rho_s(k)\rho_{s'}(n-k)|\leq \rho_{\min\{s,s'\}}(n)\quad\text{for all}\quad n\in\mathbb{Z}^3.
\end{equation*}
\end{lemma}
\begin{proof}
Since $|n-k|\geq |n|-|k|\geq |n|/2$ for $|k|\leq |n|/2$, we see that 
\begin{eqnarray*}
(\sum_{{k\in\mathbb{Z}^3}, |k|\leq |n|/2}+\sum_{{k\in\mathbb{Z}^3}, |k|>|n|/2})|\rho_s(k)\rho_{s'}(n-k)|&\leq& 
\sum_{{k\in\mathbb{Z}^3}}|\rho_s(k)\rho_{s'}(|n|/2)|+\sum_{{k\in\mathbb{Z}^3}}|\rho_s(|n|/2)\rho_{s'}(n-k)|\\
&\leq&
\max\{\rho_s(n),\rho_{s'}(n)\}.
\end{eqnarray*}
\end{proof}
\begin{remark}\label{pointwise}
We see that if $u\not \in X^{s-0}$, then there is $\{n_j\}_j$ ($n_j\to\infty$) such that $|\hat u(n_j)|\geq \rho_s(n_j)$.
Moreover, if $u\in X^s$, then $|\hat u(n)|\leq \rho_{s-3}(n)$ for all $n\in\mathbb{Z}^3$.
\end{remark}

\begin{proof}[Proof of Theorem \ref{loss of smoothness}.]
Recall that 
\begin{equation}\label{initial data}
\mathcal{F}(\mathcal L_2(t)B_1)=i\Phi_2(t,n)
\begin{pmatrix}
n_3\hat E_{0,2}(n)-n_2 \hat E_{0,3}(n)\\
n_1 \hat E_{0,3}(n)-n_3\hat  E_{0,1}(n)\\
n_2 \hat E_{0,1}(n)-n_1\hat E_{0,2}(n)
\end{pmatrix}.
\end{equation}
We focus on the second component of the vector in \eqref{initial data}. We choose $E_{0,1}$ in order to satisfy $|E_{0,1}(n)|\geq \rho_{1+\delta/2}(n)$.
In particular, this gives us that $|n_3E_{0,1}(n)|\geq \rho_{1+\delta/2}(n)$ for all $n_3\not=0$.


 Assume that $(v, B)\in C([0,T]:X^{5+\delta}(\mathbb{T}^3))$.
 We show that a loss of regularity is brought from non-smoothness of $E_0$ not $B_0$.  
 By Remark \ref{pointwise} and Proposition \ref{relation}, we have 
 $$
 |\hat v(t,n)|, |\hat B(t,n)|\leq \rho_{2+\delta-0}(n)\quad\text{with eventually}\quad
 C_1=C_1(T)\quad\text{and}\quad C_2=C_2(T).
 $$
By Lemma \ref{pointwiseconv}, we see that 
\begin{equation*}
|\mathcal F\left[v\times B\right](t,n)|\leq \rho_{2+\delta-0}(n)\quad\text{for}\quad t<T.
\end{equation*}
thus,
\begin{equation*}
|M_4(v,B)|\leq T\rho_{2+\delta-0}(n).
\end{equation*}
On the other hand, for any $n$ (with $n_3\not=3$), there is $t\in (0,T)$ (in this case we set 
$t=\frac{(2k+1)/2}{\sqrt{|n|^2-1}}\pi<T$ for some  $k\in\mathbb{Z}$) such that 
\begin{equation*}
|\Phi_2(t,n)n_{3}E_{0,1}(n)|\geq \rho_{2+\delta/2}(n).
\end{equation*}

Since $B_0$, $E_{0,2}$ and $E_{0,3}$ are smooth enough, we see that 
\begin{equation*}
|\Phi_1(t,n)(B_0(n)/2)|,
|\Phi_2(t,n)B_0(n)|,
|\Phi_2(t,n)n_1E_{0,3}(n)|
\leq \rho_{s}(n)
\end{equation*}
for any $s>0$ (in this case we set $s=2+\delta-0$) and $t\in [0,T]$.
Thus for any $n$ (with $n_3\not=0$), there is $t\in (0,T)$, we have the following lower bound for the second vector of $\hat B$:
\begin{equation*}
|\hat B_2(t,n)|\geq | \Phi_2(t,n)n_{3}\hat E_{0,1}(n)|-\rho_{2+\delta-0}(n)-T\rho_{2+\delta-0}(n)\geq 
\rho_{2+\delta/2}(n).
\end{equation*}
This is a contradiction against $B(t)\in C([0,T]:X^{5+\delta}(\mathbb{T}^3))$, since
$|\hat B(t,n)|$ must be controlled by $\rho_{2+\delta-0}(n)$ for any $t\in[0,T]$.
\end{proof}

{\bf Acknowledgments.}
The second author  thanks  the Pacific Institute for the Mathematical Sciences
 for support of his presence there during the academic
year 2010/2011.
This paper developed during a stay of the second  author as a PostDoc at the Department of Mathematics and Statistics,
University of Victoria.

\end{document}